\documentclass{amsart}

\usepackage{graphicx}

\usepackage[cmtip,all]{xy}
\usepackage{color}
\usepackage{hyperref}
\newtheorem{theorem}{Theorem}[section]
\newtheorem{proposition}[theorem]{Proposition}

\newtheorem{lemma}[theorem]{Lemma}

\newtheorem{corollary}[theorem]{Corollary}

\theoremstyle{definition}
\newtheorem{definition}[theorem]{Definition}

\theoremstyle{remark}
\newtheorem{remark}[theorem]{Remark}

\numberwithin{equation}{section}

\begin{document}

% \title[short text for running head]{full title}
%\title{}

%    Only \author and \address are required; other information is
%    optional.  Remove any unused author tags.

%    author one information
% \author[short version for running head]{name for top of paper}
\title[]{The Veronese Geometry of Dziobek Configurations and Generic Finiteness for Homogeneous Potentials}
\author[Dias]{Thiago Dias}

\address{Departamento de Matemática, Universidade Federal Rural de Pernambuco - Rua Dom Manuel de Medeiros s/n, 52171-900, Recife, Pernambuco, Brasil}
\email{thiago.diasoliveira@ufrpe.br}

%    author two information
%\author{}
%\address{}
%\curraddr{}
%\email{}
%\thanks{}

%    \subjclass is required.
\date{December 2025}

\dedicatory{}

%    Abstract is required.

\begin{abstract}

The main contribution of this paper is the proof of the generic finiteness of Dziobek central configurations for a homogeneous potential and the derivation of a uniform upper bound for their number. By exploiting the isomorphism between the Veronese variety and the determinantal variety associated with the Dziobek conditions, we define the Dziobek-Veronese variety and apply the dimension of fibers theorem to analyze the projection from the space of configurations and masses to the space of masses. We prove that the fibers of this projection, representing the central configurations for a given mass vector, are finite for masses chosen outside a proper algebraic subvariety. Furthermore, we utilize that the Dziobek variety is defined by an intersection of quadrics to obtain a bound of Bezout type for the number of Dziobek configurations with fixed masses given by a power of $2$ with exponent quadratic in $n$. Unlike previous estimates tailored to specific potentials, this bound depends solely on the dimension $n$. For instance, for the four-body problem, our bound reduces to $8192$, which is lower than the bound of $8472$ established by Moeckel and Hampton. This suggests that the complexity of counting Dziobek configurations for generic masses is governed primarily by the ambient geometry of the configuration space, rather than by the non-linearity of the interaction potential.
\end{abstract}

\dedicatory{In memory of professor Antônio Carlos.}

\maketitle

%    Text of article.

\section{Introduction}

% https://academic.oup.com/imrn/pages/ms_prep_submission

%para ajudar nessa parte

%http://www.scholarpedia.org/article/Central_configurations

% também pode ser útil:
%https://share.google/mMGtPiLGfAvxnC2wr

%\cite{smale1970topology, meyer1988introduction}.

The $n$-body problem remains a persistent challenge in celestial mechanics and dynamical systems theory. Central configurations play a crucial role in this study because they provide the initial conditions for the only known explicit solutions to the problem. Furthermore, they determine the asymptotic behavior of total collisions (\cite{sundman1909probleme} and \cite[\S 333-339 and \S361-364]{wintner2014analytical}) and characterize the bifurcations that define the topology of the integral manifolds~\cite{albouy1993integral, cabral1973integral, mccord1998integral, smale1970topology}. Central configurations are invariant modulo rotation, translation, and dilation; hence, it is natural to consider classes of central configurations modulo these transformations. In recent decades, the study of these configurations has expanded beyond the classical Newtonian case to include power-law or quasi-homogeneous potentials (see, for example, \cite{diacu2006central, tang2024perturbing, wang2023centered, liu2023stacked, santos2019inverse}). In this work, we consider the class of potentials given by $U_a=\frac{1}{2a + 2} \sum_{i < j} m_i m_j \, r_{ij}^{2a + 2}$ if $a \ne -1$, or $U_a=\sum_{i < j} m_i m_j \log r_{ij}$ if $a = -1$.

The finiteness problem,  a longstanding open question proposed in 1918 by Chazy~\cite{chazy1918certaines}, which appeared in Wintner's book~\cite{wintner2014analytical} and was famously listed by Smale as the 6th problem for the 21st century~\cite{smale1998mathematical}, concerns the finiteness of the number of classes of central configurations for a fixed set of positive masses.

The finiteness of central configurations has been established for $n=3$ by Euler~\cite{euler1767motu} and Lagrange~\cite{lagrange1772essai}. Central configurations are typically classified according to their dimension, which corresponds to the dimension of the affine hull formed by the $n$ bodies. Central configurations of dimension $1$, $2$, and $3$ are called, respectively, collinear, planar, and spatial. For an arbitrary $n$, finiteness is known to hold for collinear configurations~\cite{moulton1910straight} and when the dimension of the configuration is $n-1$~\cite{saari1980role}. Hampton and Moeckel proved finiteness for the planar four-body problem by using BKK theory~\cite{hampton2006finiteness}. Since the general case remains unsolved, significant research effort has shifted toward the question of \emph{generic finiteness}: whether the number of central configurations is finite for a generic choice of mass parameters. This property was established for the planar $5$-body problem by Albouy and Kaloshin in~\cite{albouy2012finiteness}. More precisely, for mass parameters in the complement of a codimension $2$ hypersurface of the mass space, the number of planar central configurations is finite. This result relies on codifying the behavior of unbounded singular sequences of normalized central configurations going to infinity using bicolored graphs named $zw$-diagrams. In~\cite{chang2024toward} and~\cite{chang2025toward}, Chen and Ming address the planar six-body problem and extend the Albouy-Kaloshin theory by introducing a matrix algebra that enables the determination of the $zw$-diagrams for the six-body problem. Furthermore, they establish mass relations that eliminate 61 of the 85 diagrams found. In~\cite{hampton2011finiteness}, Hampton and Jensen proved generic finiteness for spatial central configurations of the $5$-body problem. The Albouy-Kaloshin method has been successfully applied to study the finiteness problem for vortices (see e.g., \cite{yu2023finiteness} and \cite{yu2025finiteness}) and for balanced central configurations ~\cite{wang2024finiteness}.

While the cases of dimension $1$ and $n-1$ are well understood, a natural progression is the study of Dziobek configurations, defined as $n$-body configurations spanning a dimension of $n-2$. In $1900$, Otto Dziobek~\cite{dziobek1900ueber} introduced the equations known in the literature as \emph{Dziobek equations} (Eq. 12 in the original text). This work also introduced mutual distances as variables adequate to study central configurations. Dziobek's equations were seldom utilized in the first half of the 20th century. However, in works such as~\cite{meyer1988bifurcations}, Meyer and Schmidt illustrated the power of this formulation in handling otherwise intractable problems, helping to promote the use of mutual distances and Dziobek equations. Moeckel~\cite{moeckel2001generic} derived the Dziobek equations using exterior product computations, while Albouy~\cite{albouy2003paper} obtained them by studying the kernel of the configuration matrix. Recently, Leandro~\cite{leandro2025moments} generalized the Dziobek theory to central configurations of arbitrary dimension by developing the classical theory of moments of weighted points. Since their original formulation, Dziobek configurations have attracted sustained interest within the celestial mechanics community, motivating a wide range of studies such as~\cite{albouy2008symmetry, almeida2004dziobek,  leandro2003finiteness, llibre2015note, macmillan1932permanent,   santos2023bifurcations}.

Dziobek configurations have been a fruitful ground for investigating generic finiteness. Moeckel~\cite{moeckel1985relative} used algebraic geometry techniques to establish generic finiteness for the four-body case. Later, he refined these methods for $n \geq 4$ in~\cite{moeckel2001generic}. More recently, the result of generic finiteness was generalized to semi-integer potentials in~\cite{dias2017new} using the Jacobian criterion.

A key geometric insight arises when revisiting Dziobek's original work. In~\cite{dziobek1900ueber}, Dziobek presented the condition for four bodies which appears in the literature as $s_{ij}s_{kl}-s_{ik}s_{jl} = 0$. We observe that these equations correspond exactly to the quadratic relations defining the \emph{Veronese variety} of degree 2. Named after Giuseppe Veronese (1854--1917) for his work with the Veronese surface, this embedding is a fundamental tool in algebraic geometry used to map a projective space $\mathbb{P}^{n}$ into a higher-dimensional projective space. Its significance lies in its role in linearizing geometric problems, effectively allowing hypersurfaces of degree $d$ to be studied as hyperplane sections of the Veronese variety.

In this work, we extend the study of generic finiteness to a broader class of homogeneous potentials $U_a$, where the interaction depends on the mutual distances raised to the power $2a$. This framework unifies the classical Newtonian $n$-body problem ($a = -3/2$) and, in the planar case, the Helmholtz $n$-vortex problem ($a = -1$, logarithmic potential), allowing us to treat these distinct physical systems under a common geometric formalism.

The main contribution of this paper is the proof of the generic finiteness of Dziobek central configurations for the potential $U_a$. By exploiting the isomorphism between the Veronese variety and the determinantal variety associated with the Dziobek conditions, we apply the dimension of fibers Theorem to analyze the projection from the space of configurations and masses to the space of masses. We prove that the fibers of this projection, representing the central configurations for a given mass vector, are finite for masses chosen outside a proper algebraic subvariety.

A notable consequence of our geometric approach is the derivation of an explicit upper bound for the number of central configurations, specifically $2^{\binom{n+1}{2}+n-1}$. This estimate is of Bezout type, arising naturally from the intersection of quadrics that define the Veronese variety. Our formula is uniform and independent of the potential parameter $a$. We observe that for the four-body problem, our bound yields $2^{13} = 8192$, which improves upon the estimate of $8472$ established by Moeckel and Hamptom in \cite{hampton2006finiteness}. This fact suggests that the counting problem is intrinsically governed by the  geometry of the configuration space.

The paper is organized as follows. Section \ref{sec:config} introduces the potential $U_a$, the equations of motion, and the concept of central configurations expressed through the shape variables $u_{ij}$. In Section \ref{sec:dzio}, we review the fundamentals of Dziobek configurations. Subsequently, Section \ref{sec:ver} develops the algebraic geometry framework, with particular attention to the Veronese map. Finally, Section \ref{sec:fin} establishes the dimension estimates necessary to prove our main finiteness result, and we derive the explicit Bezout-type bound for the number of Dziobek configurations.

\section{Central Configuration Equations}\label{sec:config}

Consider $n$ bodies with positions $x_1,...,x_n \in \mathbb{R}^d$  with masses positive mass $m_1,...,m_n \in \mathbb{R}$. Let $r_{ij}=\|x_i - x_j\|$ be the mutual distance between the bodies $x_i$ and $x_j$. In this work, we assume $r_{ij} \neq 0$. Suppose that the bodies are moving in $\mathbb{R}^{d}$ under the action of homogeneous potential

$$ U_a(x) = 
\begin{cases}
\displaystyle \frac{1}{2a + 2} \sum_{i < j} m_i m_j \, r_{ij}^{2a + 2}, & \text{if } a \in  \text{ and } a \ne -1, \\
\displaystyle \sum_{i < j} m_i m_j \log r_{ij}, & \text{if } a = -1.
\end{cases}
$$

In the case where \( a \neq -1 \), or when \( a = -1 \) and  \( d > 2 \), the equations of motion are given by:

\begin{equation}\label{eqcc}
  \ddot{x}_{i}= \sum_{j\neq i} m_j r_{ij}^{2a}(x_j - x_i), 
\end{equation}

The Newtonian $n$-body problem corresponds to $a=-3/2.$ The case in which $a=-1$ and $d=2$ corresponds the Helmholtz $n$-vortex problem and the equations of motion are given by

\[
\dot{q}_i = -K \sum_{j \ne i} m_j r_{ij}^{-2}(x_j - x_i),
\]

where 

\begin{equation*}
K = 
\begin{bmatrix}
0 & -1 \\
1 & 0
\end{bmatrix}.
\end{equation*}

\begin{definition}\label{defcc}
The vector $(x_1,...,x_n) \in (\mathbb{R}^{d})^n$ will be called \emph{configuration}. A configuration $x$ is said to be a \emph{central configuration} associated  with  the potential $U_a$, if there exists a nonzero constant $\lambda \ne 0$ such that
\begin{equation}
\sum_{i \ne j} m_i (x_i - x_j) r_{ij}^{2a} + \lambda(x_j - c) = 0, \quad \text{for } j = 1, \dots, n,
\end{equation}
where the point $c$ is the \emph{center of mass} given by
\[
c = \frac{1}{M}(m_1 x_1 + \cdots + m_n x_n),  
\]
and $M= m_1 + \cdots + m_n \ne 0$ denotes the \emph{total mass} of the system.
\end{definition}

\begin{definition}The \emph{dimension} of a central configuration $x$, denoted by $\delta(x)$ is the dimension of the smallest space that contains $x_1,...,x_n$.
\end{definition}

%\begin{remark}
%Formally, for $a, b \in \mathbb{R}$, the power $a^b$ is defined by means of complex functions as:
%\[
 %   a^b = e^{b \operatorname{Ln}(a)} = e^{b (\ln|a| + i \arg(a) + 2\pi i p)}, \quad p \in \mathbb{Z}.
%\]
%While this definition generally yields a set of complex values, in the context of the configurations discussed here, we restrict our analysis to real-valued quantities. In our approach, the Veronese map will conceal the real power $a$; consequently, we will only need to address complex powers in Theorem \ref{prop:main}.
%\end{remark}

% https://www.youtube.com/watch?v=ZxYOEwM6Wbk
%%% https://math.stackexchange.com/questions/2790118/how-is-the-definition-for-exponentiation-extended-to-rationals-and-reals
%%%%% https://math.stackexchange.com/questions/2839430/how-to-define-ab-when-b-is-real
%%%%%% https://math.stackexchange.com/questions/2251283/formal-definition-of-numbers-with-real-exponents
It is straightforward to show that $\lambda > 0$ (see \cite[p.~4]{moeckel2014lectures}. Define $r_{0}=(M^{-1}\lambda)^{-2a}$.  $x$ is a central configuration with masses $m_1, \dots, m_n$ if, and only if, $x$ satisfies the equations
\begin{equation}\label{eqccro}
\sum_{\substack{i = 1 \\ i \ne j}}^n m_i (r_{ij}^{2a}-r_{0}^{2a}) (x_i - x_j) = 0, \quad j = 1, \dots, n.
\end{equation}

Following \cite{chen2018strictly}, we eliminate $r_0$ by defining he set of dimensionless ``shape'' variables $u = (u_{ij})_{i < j}$ is defined by

\begin{equation}\label{defshape}
u_{ij} = \frac{ r_{ij}}{r_0} ,
\end{equation}

The system \eqref{eqccro} becomes

\begin{equation}\label{eqccshape}
\sum_{i = 1}^n m_i s_{ij}x_i = 0, \quad j = 1, \dots, n,
\end{equation}

where
\begin{equation}\label{sdef}
    s_{ij}=(u_{ij}^{2a}-1) \text{ and } m_{i}s_{ii}= -\sum_{j \neq i}m_js_{ij}
\end{equation}

\section{Dziobek Central Configurations}\label{sec:dzio}

A central configuration is called a \textit{Dziobek configuration} if $\delta(x)=n-2$. In this section, we provide a self-contained presentation of the theory necessary to establish our proof of generic finiteness. For excellent expositions on Dziobek central configurations, we refer the reader to the works of Rick Moeckel in \cite{moeckel2001generic} and \cite{moeckel2014lectures} and Alain Albouy \cite{albouy2003paper}

\begin{definition}
    The matrix configuration associated to a Dziobek central configuration is the $(n-1) \times n$ matrix given by:

$$X= \begin{pmatrix}
1 & \dots & 1 \\
x_{11} & \dots & x_{1n} \\
\vdots & \dots & \vdots \\
x_{(n-2)1} & \dots & x_{(n-2)n} 
\end{pmatrix}_{(n-1) \times n}$$
\end{definition}

It is easy to see that $\text{rank}(X)=\delta(x)-1$. Hence, we can take a unique generator for the kernel of $X$.

%\begin{equation}\label{deltadef}
%n=(\delta_1,...,\delta_n),
%\end{equation}
 %for the Kernel of the configuration matrix of a Dziobek configuration. 

%Let $v_{kl} = x_1  \wedge \dots \wedge x_n $ be an $(n-2)$-dimensional exterior product, where the terms  $x_k $, and $x_l $ have been omitted. Now consider the expression $\sum_{j \neq i} m_j s_{ij} x_j \wedge v_{ikl} = 0$.

%Using the anticommutativity of the exterior product, the definition \eqref{deltyadef}, we obtain the following relation:
%\begin{equation}\label{eqmsd}
 %   m_k s_{ik}\delta_l = m_l s_{il}\delta_k.
%\end{equation}

\begin{proposition}\label{propmsd}
Let $\Delta=(\delta_1, \dots, \delta_n)$ be a vector in the kernel of the configuration matrix of a Dziobek configuration. $\Delta$ has at least two nonzero entries and the following relation holds:
\begin{equation}\label{eqmsd}
    m_k s_{ik}\delta_l = m_l s_{il}\delta_k,
\end{equation}
for $i=1,\dots,n$ and $1 \leq k<l \leq n$.
\end{proposition}

\begin{proof}
Let $v_{kl} = x_1 \wedge \dots \wedge \widehat{x}_k \wedge \dots \wedge \widehat{x}_l \wedge \dots \wedge x_n$ be an $(n-2)$-dimensional exterior product, omitting the terms $x_k$ and $x_l$. Consider the equation:
\[
\sum_{j \neq i} m_j s_{ij} x_j \wedge v_{kl} = 0.
\]

Let $\mathcal{B}=\{e_1, \dots, e_{n-2}\}$ be a basis of $\mathbb{R}^{n-2}$ and define the $(n-1)\times(n-1)$ matrix $\hat{X}_k$ obtained from $X$ by removing its $k$-th column. Denote its determinant by $|\hat{X}_k|$. Note that:
\[
x_j \wedge v_{kl}=\begin{cases}
      0, & \text{if } j\neq k,l,\\
      (-1)^{k-1}|\hat{X}_k| e_1\wedge \dots \wedge e_{n-2}, & \text{if } j=k,\\
      (-1)^{l}|\hat{X}_l| e_1\wedge \dots \wedge e_{n-2}, & \text{if } j=l.
\end{cases}
\]

Choosing $\Delta=(\delta_1, \dots, \delta_n) = (|\hat{X}_1|, \dots, (-1)^k|\hat{X}_k|, \dots, (-1)^n|\hat{X}_{n}|)$ yields the result.
\end{proof}

\begin{proposition}\label{dizeq}
Let $x$ be a central configuration with non-zero masses and $\delta(x) = n-2$, and let $s_{ij}$ be as in \eqref{sdef}. Then there exists $c \neq 0$ such that:
\begin{equation}
    m_im_js_{ij} = \kappa \delta_i \delta_j.
\end{equation}
\end{proposition}

\begin{proof}
The proof relies on the structure of the kernel of the configuration matrix $X$. First, observe that for each fixed $j$, the vector $(m_1 s_{1j}, \dots, m_n s_{nj})$ belongs to the kernel of $X$.

Since $\dim(\ker(X))=1$ and $\Delta=(\delta_1, \dots, \delta_n)$ generates the kernel of $X$, there must exist constants $k_j$ such that:
\[
m_i s_{ij} = k_j \delta_i.
\]
Invoking the symmetry $s_{ij} = s_{ji}$, it follows that the  matrix

$$\begin{pmatrix}
    k_1 & \dots & k_n \\
    \frac{\delta_1}{m_1} & \dots & \frac{\delta_n}{m_n}
\end{pmatrix}$$

has rank $1$, which leads directly to the formula:
\[
m_i m_j s_{ij} = \kappa \delta_i \delta_j.
\]
Finally, if $c=0$, all quantities $s_{ij}$ are zero, corresponding to a regular simplex. This contradicts the hypothesis that the configuration is Dziobek.
\end{proof}

As an immediate consequence, we have:

\begin{corollary}\label{corverdiz}
Let $x$ be a Dziobek configuration. Then
\begin{equation}\label{verdiz}
    s_{ij} s_{kl} = s_{il}s_{kj},
\end{equation}
where $i, j, k, l \in \{1, \dots, n\}$.
\end{corollary}

In the Section \ref{sec:ver} we will connect these equations to the classical Veronese map from algebraic geometry.  We associate a Dziobek configuration $x$ with a point $s=(s_{ij}) \in \mathbb{R}^{\binom{n+2}{2}}$, where $s_{ij}$ is given by \eqref{sdef}

\begin{corollary}
Let $s=(s_{ij})$ be the point associated with a Dziobek configuration $x$. Let $V_n$ be the determinantal variety defined by the vanishing of the $2 \times 2$ minors of the symmetric matrix
$$S=\left(\begin{array}{cccc}
s_{11}&s_{12}&\cdots&s_{1n}\\
s_{12}&s_{22}&\cdots&s_{2n}\\
\vdots&\vdots&\cdots&\vdots\\
s_{1n}&s_{2n}&\cdots&s_{nn}\\
\end{array}\right).$$
Then $s=(s_{ij}) \in V_n$.
\end{corollary}

\begin{proof}
It follows directly from equations \eqref{verdiz}.
\end{proof}

%\begin{definition}
 %   Let $x$ be a configuration and consider the vector $r = (r^2_{ij})$, $1 \leq i < j \leq n$, whose entries are given by the mutual distances $r_{ij}^2 = \|x_i - x_j\|^2$. We say that the \textit{Cayley-Menger matrix associated with configuration $x$} is the symmetric matrix defined by:
%$$
%\text{CM}(x) = \begin{pmatrix}
%0 & 1 & 1 & 1 & \dots & 1 \\
%1 & 0 & r^2_{12} & r^2_{13} & \dots & r_{1n} \\
%1 & r^2_{21} & 0 & r^2_{23} & \dots & r^2_{2n} \\
%\vdots & \vdots & \vdots & \vdots & & \vdots \\
%1 & r^2_{n1} & r^2_{n2} & r^2_{n3} & \dots & 0
%\end{pmatrix}_{(n+1)\times(n+1)},
%$
%where the %\textit{Cayley-Menger determinant} is the number $F(x) = |A(x)|$.
%\end{definition}

%\begin{corollary}Let $x$ be a configuration. Then $\mathrm{Ker}(\text{CM}(x))$ and $\mathrm{Ker}(X)$ are isomorphic through the linear map
%\begin{align*}
 %   T : \mathrm{Ker}(X) &\to \mathrm{Ker}(\text{CM}(x))) \\
  %  v = (v_1, \dots, v_n) &\mapsto \tilde{v} = \left(-\sum_{i=1}^n \|x_i\|^2 v_i, v_1, \dots, v_n\right).
%\end{align*}
%In particular $\text{dim}(\text{Ker})(CM(x))=1$
%\end{corollary}

\section{The Veronese Embedding and Dimension Theory}\label{sec:ver}

In this section, we recall properties of the Veronese map and the dimension theorems necessary for our proof. We assume familiarity with basic projective algebraic geometry, particularly the material covered in \cite[Lecture 1]{harris1992algebraic} and \cite[Chapter 1]{shafarevich1994basic}. For a concise reference gathering all definitions and properties of projective varieties used in this work, we refer the reader to \cite[Section 4]{moeckel2001generic}.

\subsection{The Veronese Map}
The Veronese map  of degree $d$ in $\mathbb{P}^n$ is the regular map $v_{d}: \mathbb{P}^n \rightarrow \mathbb{P}^{\binom{n+d}{d}-1}$ is defined by  $v_{d}([x_0,...,x_n])=[...:x^{I}:....]$, where $x^{I}$ ranges over all monomials of degree $d$ in $x_1,...,x_n$. For example, the Veronese variety of degree $2$ in $\mathbb{P}^2$ is given by $v_{2,2}: \mathbb{P}^2 \rightarrow \mathbb{P}^5$ defined by
$$v_{2,2}([x_0:x_1:x_2)=[x_{0}^2:x_{1}^2:x_{2}^2:x_1x_2:x_1x_3:x_2x_3].$$

The Veronese variety of degree $d$ in $\mathbb{P}^n$, denoted by $V_{n,d}$, is the closure of the image of the Veronese map.
It is a well-known that $V_{n,d}$ is isomorphic to $\mathbb{P}^{n}$ for all $d$. In particular, $\text{dim}V_{n,d}=n.$ and $V_{n,d}$ is a irreducible variety.

Following example $2.6$ of \cite{harris1992algebraic}, denote by $\{z_{ij}\}_{0 \leq i \leq j \leq n}$ the coordinates of $\mathbb{P}^{\binom{n+2}{2}-1}$, where we identify $z_{ij} = z_{ji}$.  The ideal $I(V_{n,2})$ is generated by the $2 \times 2$ minors of the symmetric matrix
\[
M =
\begin{pmatrix}
z_{00} & z_{01} & \cdots & z_{0n} \\
z_{01} & z_{11} & \cdots & z_{1n} \\
\vdots & \vdots & \ddots & \vdots \\
z_{0n} & z_{1n} & \cdots & z_{nn}
\end{pmatrix}.
\]
Explicitly, these generators are the quadratic binomials:
\[
z_{ij}z_{kl} - z_{ik}z_{jl} = 0, \quad \text{for } 0 \le i,j,k,l \le n.
\]

These relations are the same as satisfied by $s_{ij}$ exhibit in Corollary \ref{corverdiz}. In general, for any $n$, observe that the Veronese variety of degree $2$ in $\mathbb{P}^{n-1}$  is isomorphic to the determinantal variety $V(s)_{n}$ defined by the zero locus of equations \eqref{verdiz} containing the vectors $s_{ij}$ coming from Dziobek configurations and, in particular, has dimension $n-1$. This observation will be crucial for the approach presented in section \ref{sec:fin}.

%The Veronese map of degree $d$, $\nu_{d}: \mathbb{P}^n \rightarrow \mathbb{P}^N$ with $N = \binom{n+d}{d}-1$, is defined by the complete linear system of hypersurfaces of degree $d$. For the case $d=2$, the map $\nu_2: \mathbb{P}^{n-1} \to \mathbb{P}^{\binom{n+1}{2}-1}$ embeds the projective space as a variety defined by quadratic relations.
%Specifically, the image of the Veronese map (the \emph{Veronese variety}) is the locus of points satisfying the determinantal conditions $z_{ik}z_{jl} - z_{il}z_{jk} = 0$. As we observed in Section \ref{sec:dzio}, these are exactly the relations satisfied by the shape variables $s_{ij}$ of Dziobek configurations. Since $\nu_d$ is an embedding, the Veronese variety is isomorphic to $\mathbb{P}^{n-1}$ 

\subsection{Fiber Dimension}
Our main result relies on establishing that the set of Dziobek configurations associated to generic masses consists of isolated points. To do this, we employ the classical theorem on the dimension of fibers of morphisms.

\begin{theorem}[{Dimension of Fibers Theorem, \cite[Thm. 1.25]{shafarevich1994basic}} ] \label{fiberdim}
Let $f : X \to Y$ be a regular map between irreducible algebraic varieties. Suppose that $f$ is surjective: $f(X) = Y$, and that $\dim X = n$, $\dim Y = m$. Then $m \le n$, and
\begin{enumerate}
    \item[(i)] $\dim F \ge n - m$ for any $y \in Y$ and for any component $F$ of the fibre $f^{-1}(y)$;
    \item[(ii)] there exists a nonempty open subset $U \subset Y$ such that $\dim f^{-1}(y) = n - m$ for $y \in U$.
\end{enumerate}
\end{theorem}

In Section \ref{sec:fin}, we introduce a regular map to $V_n(s)$ and determine the fiber dimension at points $(s_{ij})$ associated with Dziobek configurations. Our goal is to identify a quasi-projective variety containing all Dziobek configurations. To do this, we discard the closed set of points where the fiber dimension exceeds that of a generic Dziobek configuration. We rely on the semicontinuity of fiber dimensions to achieve this result.

\begin{theorem}[Semicontinuity of the Fibers \cite{mumford2004red}]
\label{thm:semicontinuity}
Let $f \colon V \to W$ be a regular map between irreducible algebraic varieties. For all $x \in V$, define
\[
e(x) = \max \{ \dim(Z) \mid Z \text{ is a component of } f^{-1}(f(x)) \text{ with } x \in Z \}.
\]
Then $e(x)$ is upper semi-continuous. That is, for all integers $m \ge 0$, the set
\[
S_m(f) := \{ x \in V \mid e(x) \ge m \}
\]
is closed in $V$.
\end{theorem}

Since the domain $V$ in our application implies a decomposition into irreducible components, we need to extend the previous result to reducible varieties. The following lemma ensures that we can apply the semicontinuity theorem component-wise.

\begin{proposition}
\label{prop:reducible_semicontinuity}
Let $f \colon V \to W$ be a morphism of algebraic varieties with $W$ irreducible. Suppose that
\[
V = V_1 \cup \cdots \cup V_k
\]
is the decomposition of $V$ into irreducible components, and let $f_i := f|_{V_i}$. Then
\[
S_m(f) = \bigcup_{i=1}^k S_m(f_i).
\]

In particular, $S_m(f)$ is closed.
\end{proposition}

\begin{proof}
First, we assert that the total fiber $f^{-1}(y)$ is the union of the restricted fibers $f_i^{-1}(y)$ for any $y$ in $W$. This follows from the fact that $f_i$ is the restriction of $f$ to $V_i$, which implies that $f_i^{-1}(y)$ is equal to the intersection $f^{-1}(y) \cap V_i$.

To prove the inclusion $\bigcup_{i=1}^k S_m(f_i) \subseteq S_m(f)$, let $y$ be an element of $S_m(f_j)$ for some index $j$. By definition, there exists an irreducible component $Z$ of $f_j^{-1}(y)$ such that $\dim(Z) \ge m$. Since the restricted fiber $f_j^{-1}(y)$ is a subset of the total fiber $f^{-1}(y)$, the set $Z$ is also a closed subset of $f^{-1}(y)$. Consequently, $Z$ is contained in some irreducible component of the total fiber with dimension at least $m$. This implies that $e(y) \ge m$ and thus $y \in S_m(f)$.

Conversely, let $y \in S_m(f)$. There exists an irreducible component $Z$ of the fiber $f^{-1}(y)$ passing through $y$ with $\dim(Z) \ge m$. As $V$ is the union of the irreducible components $V_i$, the irreducible set $Z$ must be contained entirely within one specific component $V_j$. Therefore, $Z$ is a subset of the intersection $f^{-1}(y) \cap V_j$, which is exactly the restricted fiber $f_j^{-1}(y)$. Since $Z$ is a maximal irreducible subset of the total fiber, it is also a maximal irreducible subset of the restricted fiber. We conclude that $y \in S_m(f_j)$, which completes the proof.
\end{proof}

\section{Generic Finiteness of Dziobek Central Configurations}\label{sec:fin}

In this section, we present the proof of the generic finiteness of Dziobek central configurations for the potential $U_a$. Our approach relies on the geometric identification established in Section \ref{sec:dzio}, where we showed that the configuration space is isomorphic to the Veronese variety of degree 2.

We construct an projective variety $W$  that encodes the Dziobek conditions derived in Section \ref{sec:dzio}. By exploiting an isomorphism with the Veronese variety and applying the Fiber Dimension Theorem to the projection from $V$ onto the space of masses, we demonstrate that the set of Dziobek configurations for a generic choice of mass parameters is finite.

%In particular, there is a veronese map from the space of normalized Plucker relations to the space of 

%$$\begin{array}{cccc}
%v_{2}:& \mathbb{P}^{n-1} &\rightarrow& \mathbb{P}^{n(n-1)/2}\\
 %         &(z_{1},...,z_{n})&\mapsto & (S_{11},...,S_{1n},...,S_{1n},...,S_{nn})
%\end{array}$$

%Now, we will define the space of masses $\mathbb{P}^{n-1}$ whose elements are $[m_1:...:m_n]$.

%Since $V(S)$ and $\mathbb{P}^{n-1}$ have the same dimension, it is expected that the fibers are finite. We need to define a suitable regular map.

Define $N=\binom{n+1}{2}$ and consider the product projective space $\mathbb{P}^{N-1}\times \mathbb{P}^{n-1}\times \mathbb{P}^{n-1}$ denote its points by $([s_{ij}],[\delta_i],[m_j])$. Define the following algebraic variety:

$$W=\{([s_{ij}],[\delta_i],[m_i]) \in \mathbb{P}^{N-1}\times \mathbb{P}^{n-1}\times \mathbb{P}^{n-1}, \eqref{eqmsd} \text{ and } \eqref{verdiz}  \text{ hold }\}.$$

The equations $\eqref{eqmsd}$ and  $\eqref{verdiz}$  are tri-homogeneous and define a projective variety $W$ of the product projective space $\mathbb{P}^{N-1}\times \mathbb{P}^{n-1}\times \mathbb{P}^{n-1}$. Moreover, by proposition \ref{propmsd} and corollary \ref{corverdiz}, $W$ contains all points $P(x)=([s_{ij}],[\delta_i],[m_i])$ provided by Dziobek configurations. Let $W=W_{1}\cup ... \cup W_{k}$ the decomposition of $V$ in irreducible components.

Consider the determinantal variety $V_n(s)$. Since this variety is defined by the homogeneous quadrics $s_{ij}s_{kl}-s_{ik}s_{jl}$, we will see $V_n(s)$ as a projective variety of $\mathbb{P}^{N-1}$. Since $V_n(S)$ and the Veronese variety $V_{n-1,2}$ are isomorphic, we have that $V_n(s)$ is an irreducible projective variety of dimension $n-1$.

\begin{definition}
Let $x$ be a Dziobek central configuration and $s_{ij}$ as in \eqref{sdef}. We define the \emph{degeneracy index} of $x$ as the cardinality of the set $\{j \mid s_{jj}=0 \}$. We say that a Dziobek central configuration is \emph{non-degenerate} if its degeneracy index is zero.
\end{definition}

\begin{remark}
We observe that proposition \ref{propmsd} the degeneracy index of a Dziobek configuration is less than $n-2$.     
\end{remark}

%Consider the projection  $\pi_{1}: W \rightarrow V_{n}(s)$ given  by $p_{1}([s_{ij}],[\delta_{i}],[m_{i}])=[s_{ij}]$ and its restrictions $p_{1}|_{W_{i}}$ to the irreducible components $W_i$ of $W$. When there is no risk of confusion, we denote $\pi_i = p_{1}|_{W_{i}}$. 

 Consider the family of projections defined as the maps $\pi_{1}^{(l)}: W \to \mathbb{P}^{\binom{l+2}{2}}$ given by
\[
\pi_{1}^{(l)}\big( ([s_{ij}], [\delta_k], [m_k]) \big) =  [s_{12} : s_{13} : \dots : s_{(n-l)(n-l)}]
\]
For simplicity, we also denote $\pi_{1}^{(0)} = \pi_1$.

Let $x$ be a Dziobek configuration with degeneracy index $0 \leq l \leq n-2$ and let $P(x)=([s_{ij}], [\delta_k], [m_k])$ the point associated with $x$.  We can suppose without loss of generality that  
 $s_{ii} \neq 0$ for $i=1, \dots, n-l$, and $s_{ii} = 0$ for $i=n-l+1, \dots, n$. Proposition \ref{propmsd} implies that $n-l \geq 2$. Equations $\eqref{verdiz}$ $s_{ii}s_{jj}-s_{ij}^2$ implies that $s_{ij}=0$ if $i$ or $j$ belong to $\{n-l+1, \dots, n\}$, and $s_{ij} \neq 0$ if $i$ and $j$ belong to $\{1,....,n-l\}$. 
 
 We call the set $E=\{(i,j): s_{ij} \neq 0 \}$ by \emph{non-degenerate block} of $x$. The set  $D=\{(i,j): s_{ij} \neq 0 \}$ will be called \emph{degenerate block} of $x$.

 Denote by $s^{n-l}$ the point of $\mathbb{P}^{\binom{n-l+2}{2}}$ whose the coordinates are $s_{ij}$ with $(i,j)$ in the non-degenerate block. $s^{n-l}$ will be called the \emph{non-degenerate coordinates} of $x$. Note that $\pi^{l}(P(x))=s^{n-l}$.

\begin{lemma}
The image of the projection $\pi_{1}^{n-l}$ is isomorphic to the Veronese variety $V_{2,n-l-1}$
\end{lemma}
\begin{proof}
By definition of $W$ we have $\text{Im}(\pi_{1}^{l}(W))\subset V_{n-l-1}(s)$. On the other hand, given a point
$[s_{12} : s_{13} : \dots : s_{(n-l)(n-l)}] \in V_{2,n-l-1}(s)$ choose an arbitrary point $[\overline{m}]=[m_1:\dots: m_{n-l}:0\dots:0]= \in \mathbb{P}^{n-1}$. Choosing a point $[\overline{\delta}]$ such that $[\overline{\delta}]=[\delta_1: \dots :\delta_{n-l}:0 \ldots:0]=[m_1s_{11}:\dots:s_{1(n-l)}:0 \dots:0]$ we have that the point $P=([s_{ij}], [\delta_k], [m_k]) \in W$ satisfies $\pi_{1}^{l}(P)=s^{n-l}$. Since $V_{n-l-1}(s)$ is isomorphic to $V_{2,n-l-1}$ the proof is complete.

\end{proof}

The subsequent lemma is a consequence of Dziobek equations \eqref{dizeq}. It will explain the Veronese geometry of Dziobek configuration in terms of degeneracy 

\begin{lemma} \label{lemgeoverdiz}
If $P(x)=([s_{ij}], [\delta_k], [m_k])$ is a point of $\mathbb{P}^{N-1}\times \mathbb{P}^{n-1}\times \mathbb{P}^{n-1}$  associated with a Dziobek configuration with degeneracy index $0 \leq l \leq n-2$, then $s^{n-l}$ is in the image of the Veronese map $v_{2,n-l-1}$. More precisely 

$$s^{n-l}=v_{2,n-l-1}\left( \left[ \sqrt{\kappa} \cdot \frac{\delta_1}{m_1}: \dots : \sqrt{\kappa} \cdot \frac{\delta_{n-l}}{m_{n-l}} \right] \right),$$
for which $\kappa$ as in Proposition \ref{corverdiz}
\end{lemma}

\begin{proof}
Note that

\begin{align*}
    v_{2,n-l-1}\left( \left[ \sqrt{\kappa} \cdot \frac{\delta_1}{m_1}: \dots : \sqrt{\kappa} \cdot \frac{\delta_{n-l}}{m_{n-l}} \right] \right)& =
    (k\frac{\delta_1^2}{m_1^2}: \kappa \frac{\delta_1}{m_1}\frac{\delta_2}{m_2}: \dots : \frac{\delta_{n-l}^2}{m_{n-l}})\\
    &=[s_{11}:s_{12}: \dots: s_{(n-l)(n-l)}]=s^{n-l}.
\end{align*}

\end{proof}

Next, we will compute the fiber of $\pi_1$ in a point $P(x)$ associated to a non-degenerate Dziobek configuration. 

\begin{proposition}\label{nondencase}
Let $x$ be a nondegenerate  Dziobek configuration and let $P(x) \in W_i$ be the point associated with a Dziobek configuration. Then the fiber $\pi_1^{-1}(P(x))$ is a unique point. In particular, $\pi_1^{-1}(P(x))$ have dimension $0$.
\end{proposition}

\begin{proof}
Let $x$ be a Dziobek central configuration with masses $m_1, \dots, m_n$. $x$ is associeted to a point
\[
P(x) = ([s_{ij}], [\delta_i], [m_i]) \in W,
\]
where $s_{ij} = u_{ij}^{2a} - 1$ and   $\Delta=(\delta_1,...,\delta_n)$ is a generator of the kernel of the configuration matrix $X$.

First, consider the case in which $x$ is nondegenerate. We observe that $ \pi_1(P(x))=[s_{ij}]\in V_{n}(s)$. We claim that the fiber $\pi_1^{-1}([s_{ij}])$ is a unique point.

%If $P(x) \in W_i$ for some irreducible component of $W$ then consider the projection $\pi_1$ restricted to $W_i$. 

By Lemma \ref{lemgeoverdiz}

\[
[s_{ij}] = v_{2,n-1}\left( \left[ \sqrt{\kappa} \cdot \frac{\delta_i}{m_i} \right] \right),
\]
which shows that $[s_{ij}]$ lies on the Veronese variety $V_{2,n-1}$. Since the Veronese map is an isomorphism onto its image, there exists a unique point $[z_1 : \dots : z_n] \in \mathbb{P}^{n-1}$ such that
\[
v_2([z_1 : \dots : z_n]) = [s_{ij}].
\]

Thus, for any point $([s_{ij}], [\delta_j], [m_j]) \in \pi_1^{-1}([s_{ij}])$, we must have
\[
\left[ \frac{\delta_i}{m_i} \right] = [z_i],
\]
where the projective coordinates $[z_i] \in \mathbb{P}^{n-1}$ are uniquely determined.

Now suppose that both $([s_{ij}], [\delta_j], [m_j])$ and $([s_{ij}], [\tilde{\delta}_j], [\tilde{m}_j])$ lie in the fiber $\pi_1^{-1}([s_{ij}])$. Note that the following matrices have rank $1$:

$$\begin{pmatrix}
    m_1s_{11} & \dots & m_ns_{1n} \\
    \delta_1 & \dots & \delta_n{m_n}
\end{pmatrix}  \qquad \text{and} \qquad \begin{pmatrix}
    m_1s_{11} & \dots & m_ns_{1n} \\
     \overline{\delta}_1 & \dots &  \overline{\delta}_n
\end{pmatrix}  $$

Since $(m_1s_{11},  \dots  ,m_ns_{1n})$ belongs to the Kernel of the configuration matrix $X$ and $\text{dim}(\text{Ker}(X))=1$ then $(\delta_1, \dots, \delta_n)$ and $(\overline{\delta}_1, \dots, \overline{\delta}_n)$ lie in the kernel of $X$ and there exists $\lambda \in \mathbb{C}\setminus\{0\}$ such that
\[
\delta_i = \lambda \overline{\delta}_i \quad \text{for all } i.
\]

Therefore,
\[
\left[ \frac{\delta_i}{m_i} \right] = [z_i] = \left[ \frac{\overline{\delta}_i}{\overline{m}_i} \right],
\]
so there exist nonzero scalars $\tau, \overline{\tau} \in \mathbb{C}$ such that
\[
\delta_i = \tau m_i z_i, \quad \overline{\delta}_i = \overline{\tau} \overline{m}_i z_i.
\]

Combining these, we find
\[
\overline{m}_i = \overline{\tau}^{-1} z_i^{-1} \overline{\delta}_i = \lambda \overline{\tau}^{-1} z_i^{-1} \delta_i = \lambda \overline{\tau}^{-1} \tau m_i,
\]
for all $i = 1, \dots, n$. Consequently,
\[
[m_i] = [\overline{m}_i] \in \mathbb{P}^{n-1}, \quad [\delta_i] = [\overline{\delta}_i] \in \mathbb{P}^{n-1}.
\]

Hence,
\[
([s_{ij}], [\delta_j], [m_j]) = ([s_{ij}], [\tilde{\delta}_j], [\tilde{m}_j]),
\]
which implies that the fiber $\pi_1^{-1}([s_{ij}])$ contains exactly one point.
\end{proof}

\begin{proposition}\label{dencase}
Let $x$ be a Dziobek configuration with degeneracy index $0< l \leq n-2$ and let $P(x)$ be the point associated with this configuration. Then the fiber $\pi_1^{-1}(\pi_{1}^{(l)}(P(x)))$ has dimension $l$.
\end{proposition}

\begin{proof}

By proposition \ref{lemgeoverdiz} the image $\text{Im}(\pi_{1}^{(l)}) \cong V_{n-l-1}(s)$ is isomorphic to the Veronese variety $V_{2,n-l-1}$. If  $P(x)= ([s_{ij}], [\delta_k], [m_k])$ is a point associated with a Dziobek configuration with degeneracy index $l$ take $s^{n-l}\in V_{n-l-1}(s)$ the non-degenerate coordinates of $x$

Proceeding analogously to the proof of Proposition \ref{nondencase}, we deduce that the classes of masses $[m_1: \dots : m_{n-l}]$ and constants $[\delta_1 : \dots : \delta_{n-l}]$ are uniquely determined for $s^{n-l}$.

However, for the indices $(i,j)$ that belong to the degenerate block, conditions $s_{ij}=0$ imply trivial constraints on the associated masses and constants in the fiber. Specifically, the masses $m_{n-l+1}, \dots, m_n$ are not constrained by the coordinates of $s^{n-l}$. Since there are $l$ such indices, and they contribute freely to the fiber dimension relative to the fixed projection $s^{n-l}$, we conclude that
\[
\dim \left( (\pi_1^{(l)})^{-1}(s^{n-l}) \right) = l.
\]
\end{proof}

\begin{proposition}\label{dimfiber}
If $W_i$ is a irreducible component of $W$ that contains a point $P(x)$ associated with a Dziobek configuration $x$, then $\text{dim}(W_i)\leq n-1.$
\end{proposition}
\begin{proof}

 Let $P(x)$ a point associated with a Dziobek central configuration with degeneracy index $t \geq 0$. Proposition \ref{nondencase} or \ref{dencase} shows that $\pi_1^{-1}(P(x))$ have dimension at most $t$. By the dimension of fibers theorem, \ref{dimfiber} (i), applyed to the morphism $\pi_1^{(l)}$ we have $l \geq \text{dim}(W_i)-\text{dim}(V_{n-l}(s))$. Since $\text{dim}(V_{n-l}(s))=n-l-1$, the result follows.
 
\end{proof}

Let $\pi_3: \Gamma \to \mathbb{P}^{n-1}$ denote the projection onto the mass space.  Following Moeckel \cite{moeckel2001generic}, we define Dziobek components. We utilize the upper semi-continuity of the fiber dimension of $\pi_3$ to identify the irreducible components of $W$ that contain a Dziobek configuration.

 For each component of $W$ we will remove a suitable closed set that does not include Dziobek central configurations. If $\text{dim}(\overline{\pi_3(W_i)})=u$ with $u \in\{1,...,n-1\}$ define $B_i=S_{n-u}(\pi_3|_{W_i})$. Set
 $B=\cup_{i=1}^{k}B_i$, and $\Gamma_i=W_i \setminus B_i$. Consider the quasi-projective 

$$\Gamma= V\setminus B=\cup_{i=1}^{k} (V_{i} \setminus B_i)=\cup_{i=1}^k \Gamma_i$$

We call $\Gamma$ the \emph{Dziobek-Veronese variety}. Note that the $\Gamma_i$ are the irreducible components of $\Gamma.$ We say that a irreducible component of $\Gamma_i$ is dominant if the closure of $p_3(\Gamma_i)$ is $\mathbb{P}^{n-1}$.

\begin{lemma} 
  $\Gamma$ contains all point $P(x)$ associated with Dziobek central configurations.
\end{lemma}

\begin{proof}
Suppose by contradiction that some point $P(x)$ does not belongs to $\Gamma$. Then there exists some indices $i_0 \in \{1.,,,k\}$ and $u_0 \in \{1,\dots n-1\}$ such that $P(x) \in W_i \cap B_i$,  $\text{dim}(\overline{\pi_3(W_i)})=u_0$ and 
$\text{dim}(\pi_3^{-1}((\pi_3(P(x)))) \geq n-u_0$. By the dimension of fibers Theorem, \ref{dimfiber}  (ii), there exists integer $r \geq 0$ such that

$$r=\text{dim}(W_i)-\text{dim}(\overline{\pi_3(W_i)})$$

$\text{dim}(\pi_3^{-1}((\pi_3(P(x)))) \geq n-u_0$ implies $r \geq n-u_0$. Hence

$$\text{dim}(W_i) \geq n-u_0+\text{dim}(\overline{\pi_3(W_i)})=n.$$

By proposition \ref{dimfiber}, we have a contradition.
\end{proof}

\begin{lemma}
  The Dziobek-Veronese variety $\Gamma$ has dimension $n-1$.  
\end{lemma}
\begin{proof}
let $\Gamma_i$ a irreducible component of $\Gamma$. Then $W_i \not \subset B_i$ and there exists indice  $u_0 \in \{1,\dots n-1\}$ with  $\text{dim}(\overline{\pi_3(W_i)})=u_0$ and 
$\text{dim}(\pi_3^{-1}((\pi_3(P)) < n-u_0$ for some $P \in W_i$.

 By the dimension of fibers Theorem, \ref{dimfiber}  (ii), there exists integer $r \geq 0$ such that

$$r=\text{dim}(W_i)-\text{dim}(\overline{\pi_3(W_i)})$$

$\text{dim}(\pi_3^{-1}((\pi_3(P)) < n-u_0$ implies $r < n-u_0$. Hence

$$\text{dim}(W_i) < n-u_0+\text{dim}(\overline{\pi_3(W_i)})=n.$$
\end{proof}

We now proceed to prove that the restriction of $\pi_3$ to each irreducible component $\Gamma_i$ of $\Gamma$ has finite fibers.

\begin{proposition} \label{fincomplex}
   There exists a proper subvriety $B \subset \mathbb{P}^{n}$ such that if $[m] \in  \mathbb{P}^{n} \setminus B$ then the fiber $\pi_{3}^{-1}([m])$ is finite. 
\end{proposition}
\begin{proof} The dimension of each component of $\Gamma$ is less than $n-1$, the dimension of fibers theorem implies that, for a generic choice of mass \( [m_i] \in \mathbb{P}^{n} \), the fibers of the restriction of the projection of $\pi_3$ to each irreducible component of $W$ are either finite, when $\Gamma_i$ is dominant, or empty, otherwise. This completes the argument.
\end{proof}

\begin{theorem}\label{prop:main}
There is a proper $C \subset \mathbb{P}_{\mathbb{R}}^{n}$ such that if $[m] \in \mathbb{P}_{\mathbb{R}}^{n}\setminus C$  there is a finite number of class of Dziobek central configurations with potential $U_{a}$ associated to $[m]$. 
\end{theorem}
\begin{proof}
If $B$ is a proper subvariety of $\mathbb{P}^{n-1}$ is as in proposition \ref{fincomplex} then $C=B \subset \mathbb{P}_{\mathbb{R}}^n$ is a proper subvariety of $\mathbb{P}_{\mathbb{R}}^{n}$. Given a point $[m] \in \mathbb{P}_{\mathbb{R}}^{n} \setminus C$ there is a finite number of points $([s_{ij}],[\delta_i],[m_i])$ in the fiber of $\pi_{3}([m])$. From the equations \eqref{sdef} we get
$$u_{ij}^{2a}= 1 + s_{ij}.$$

Taking the complex $2a$-th root we get
$$u_{ij}=\left(1+s_{ij}\right)^{1/2a}.$$
Note that a point $([s_{ij}],[\delta_i],[m_i])$ determines a finite number of points $[u_{ij}] \in \mathbb{P}^{\binom{n}{2}-1}$. When all components of $[u_{ij}]$ are real and positive its components are the mutual distances associated with a Dziobek central configuration. Since $[u_{ij}]$ determine the mutual distances up to scaling we obtain the result. 
\end{proof}

Finally, we provide a bound for the number of Dziobek configurations with $n$ bodies that does not depend on the particular choice of the potential $U_a$. To this end, we apply the following result obtained by Wallach in \cite{wallach1996theorem}, which estimates the number of irreducible components of an algebraic variety.

\begin{theorem}[{\cite[Thm. 7.1]{wallach1996theorem}}] \label{genbound}
Let $X$ be a (Zariski) closed set in $\mathbb{A}^N$ (resp. $\mathbb{P}^N$) given as the zero locus of polynomials $f_1, \dots, f_m$ (resp. homogeneous) of degree at most $k$. Then the number of irreducible components of $X$ is at most $k^N$.
\end{theorem}

\begin{theorem}\label{dzio}
If $n \geq 4$, for every $m = (m_1, \dots, m_n) \in \mathbb{R}^n \setminus B$, the number of Dziobek configurations with masses $m_1, \dots, m_n$ and potential $U_a$, $a \in \mathbb{R}\setminus \{0\}$, is less than or equal to $2^{\binom{n+1}{2}+n-1}$.
\end{theorem}

\begin{proof}
    Consider the polynomial ring $$R=\mathbb{C}[S_{11}, S_{12}, \dots, S_{nn}, \Delta_1, \dots, \Delta_n].$$
    To a Dziobek configuration $x$ with fixed masses $m_1, \dots, m_n$, we associate a point $P_{m}=([s_{ij}], [\delta_k])=(s_{11}: \dots, s_{nn}: \delta_1: \dots: \delta_n) \in \mathbb{P}^{\binom{n+1}{2}+n-1}$. This point lies in the zero locus of the algebraic variety defined by the following polynomials:
    \begin{align}
        g_{ikl} &= m_k S_{ik}\Delta_l - m_l S_{il}\Delta_k, \qquad k < l, \\
        h_{ijkl} &= S_{ij}S_{kl} - S_{ik}S_{jl}, \qquad i < j \text{ and } k < l.
    \end{align}

    For a choice of mass $m=[m_1:\dots:m_n] \in \mathbb{P}^{n-1} \setminus B$, where $B$ is defined as in Proposition \ref{fincomplex}, we define the projective variety:
    $$W_m = \left\{ ([s_{ij}],[\delta_k]) \in \mathbb{P}^{\binom{n+1}{2}+n-1} : g_{ikl}=0 \text{ and } h_{ijkl}=0 \right\}.$$

    Since $W_m$ is finite (Proposition \ref{fincomplex}) and is defined by quadratic polynomials, by Theorem \ref{genbound} the number of irreducible components  is bounded by $2^{\dim(\mathbb{P})}$, specifically $2^{\binom{n+1}{2}+n-1}$.
    
    Note that every point $P_m$ corresponding to a Dziobek configuration $x$ is indeed an isolated point. Otherwise, $P_m$ would belong to a fiber of $\pi_3$ with dimension greater than $0$, implying that $P(x)$ belongs to a component of $\Gamma$ with dimension greater than $n-1$. By Proposition \ref{dimfiber}, this is impossible. Hence, the number of points $P_m$ is bounded by $2^{\binom{n+1}{2}+n-1}$. Note that this bound does not depends on the choice of masses neither the potential $U_a$.

    Finally, the mutual distances $r_{ij}$ associated with Dziobek configurations are recovered from the algebraic coordinates $P_m=([s_{ij}], [\delta_k]) \in W_m$ via the relation:
    $$ r_{ij} = r_{0}(1 + s_{ij})^{\frac{1}{2a}}. $$
    Since we are exclusively interested in physical configurations, we require $r_{ij} \in \mathbb{R}_{>0}$. For $y \in \mathbb{C}$, the equation $y^{2a} = 1 + s_{ij}$ yields at most one positive real solution. Therefore, the number of mutual distance points $[r_{12}: \dots: r_{(n-1)n}]$ is also bounded by $2^{\binom{n+1}{2}+n-1}$. Since each mutual distance point determines a class of Dziobek configurations, the result is proved.
\end{proof}

As a particular case of  we are able to improve the bound $8472$ obtained by Moeckel and Hamptom \cite{hampton2006finiteness} for the number of  Dziobek configurations of the newtonina four body problem   for the case   for generic masses and general potential $U_a$. 
 
\begin{corollary}
    If $n=4$, there exists a closed algebraic set $B$ for which if $m = (m_1, \dots, m_4) \in \mathbb{R}^4 \setminus B$ the number of Dziobek configurations of the four body problem is $2^{13}$.
\end{corollary}

We conclude by noting that our estimate improves upon previous bounds, specifically those established by Llibre \cite{llibre1990number} for the Newtonian four- and five-body problems, Moeckel \cite{moeckel2001generic} for the Newtonian case with general $n \geq 4$, and Dias \cite{dias2017new} for semi-integer potentials. Furthermore, our result is a bound of Bezout type. Since the Dziobek-Veronese variety is defined by an intersection of quadrics related to the structure of the Veronese variety, the resulting bound is relatively small. This suggests that the complexity of the Dziobek problem is governed primarily by the ambient geometry of the configuration space, rather than by the specific homogeneous potential.

\section*{Acknowledgments}

The author would like to thank Eduardo Leandro for proposing this research topic and for the invitation to give a series of talks on this subject at the Mathematical Physics Seminar at the Department of Mathematics, UFPE. The author also thanks Rafael Holanda and Alan Muniz, organizers of the Sagui - Commutative Algebra and Algebraic Geometry Seminar, for the invitation to present the ideas that resulted in this work. Finally, the author thanks Marcelo Pedro, Michelle Gonzaga, Gabriel Bastos, João Lemos, and Vinicius Portella for their encouragement and support during the preparation of this manuscript.

%    Bibliographies can be prepared with BibTeX using amsplain,
%    amsalpha, or (for "historical" overviews) natbib style.
\bibliographystyle{plain}
%    Insert the bibliography data here.

\bibliography{refs}

\end{document}